\documentclass{elsarticle}

\usepackage[utf8]{inputenc}
\usepackage[english]{babel}
\usepackage{graphicx}

\usepackage{hyperref}
\usepackage{amsmath,amsthm,amssymb}
\usepackage{enumerate}
\usepackage{color}
\usepackage{cleveref} 
\usepackage{tikz}

\def\A{\mathcal{A}}
\def\B{\mathcal{B}}

\def\uu{\mathbf{u}}

\def\N{\mathbb{N}}
\def\R{\mathbb{R}}

\newtheorem{theorem}[]{Theorem}
\newtheorem{corollary}[theorem]{Corollary}
\newtheorem{lemma}[theorem]{Lemma}

\newtheorem{proposition}[theorem]{Proposition}
\newtheorem{definition}[theorem]{Definition}

\theoremstyle{remark}
\newtheorem{remark}[theorem]{Remark}
\newtheorem{example}[theorem]{Example}

\crefname{theorem}{Theorem}{Theorems}
\crefname{corollary}{Corollary}{Corollaries}
\crefname{example}{Example}{Examples}
\crefname{lemma}{Lemma}{Lemmas}
\crefname{proposition}{Proposition}{Propositions}
\crefname{definition}{Definition}{Definitions}
\crefname{example}{example}{examples}

\begin{document}
\begin{frontmatter}



\title{Characterization of circular D0L systems}


\author[fit]{Karel Klouda}
\author[fit]{Štěpán Starosta}

\address[fit]{Faculty of Information Technology, Czech Technical University in Prague, Thákurova~9, 160~00, Prague, Czech Republic}

\begin{abstract}
	We prove that every non-circular D0L system contains arbitrarily long repetitions.
	This result was already published in 1993 by Mignosi and Séébold, however their proof is only a sketch.
	We give here a complete proof.
	Further, employing our previous result, we give a simple algorithm to test circularity of an injective D0L system.
\end{abstract}


\end{frontmatter}

\section{Introduction}

In formal language theory, D0L languages form an important class.
See for instance\cite{RoSa80}. Starting by the work of Axel Thue, repetitions in various languages were studied.
In \cite{EhRo83}, the authors show that is it is decidable whether a D0L language is $k$-power free, i.e., does not contain a repetition of $k$ same words for some $k \in \N$. In \cite{MiSe}, the authors show that if a PD0L language is $k$-power free for some integer $k$, then it is circular. However, the authors give mostly only sketches of proofs, thus we give a sound proof here. Moreover, we generalize the result as we prove it for non-injective PD0L-systems and slightly relaxed definition of circularity, called weak circularity. We also give a simple algorithm to test whether an injective D0L system is circular.

\section{Preliminaries}

Let $\A$ be an \textit{alphabet}: a finite set of \textit{letters}.
The free monoid $\A^*$ is the set of all finite words over $\A$ endowed with concatenation.
The \textit{empty word} is denoted $\varepsilon$ and the set of all non-empty words over $\A$ is denoted $\A^+$.
The \textit{length} of $w \in \A^*$ is denoted $|w|$. Given a word $w \in \A^*$, we say that $u \in \A^*$ is a \textit{factor} of $w$ if there exists words $p$ and $s$, possibly empty, such that $w = pus$. Such a word $p$ is a \textit{prefix} of $w$, and the word $s$ is a \textit{suffix} of $w$. If $|p| < |w|$, $p$ is a \textit{proper} prefix, if $|s| < |w|$, $s$ is a \textit{proper} suffix.

The set $\A^\N$ is the set of all \textit{infinite words} over $\A$.
Given a word $w$, by $w^\omega$ we denote the infinite word $www \cdots$.

Let $\varphi$ be an endomorphism of $\A^*$.
We define
$$
\| \varphi \| = \max \{ |\varphi(a) | \colon a \in \A \} \quad \text{ and } \quad | \varphi | = \min \{ |\varphi(a) | \colon a \in \A \}.
$$

A triplet $G = (\A, \varphi, w)$ is a \textit{D0L system} if $\A$ is an alphabet, $\varphi$ is an endomorphism of $\A^*$, and $w \in \A^*$.
The word $w$ is the \textit{axiom} of $G$.
The \textit{sequence of $G$} is $E(G) = (w_i)_{i \geq 0}$ where $w_0 = w$ and $w_i = \varphi^i(w_0)$. The \textit{language of $G$} is the set $L(G) = \{ \varphi^n(w) \colon n \in \N \}$ and by $S(L(G))$ we denote the set of all factors appearing in $L(G)$. The alphabet is always considered to be the minimal alphabet necessary, i.e., $\A \cap S(L(G)) = \A$.

We say that a D0L system $G = (\A, \varphi, w)$ is \textit{injective} if for every $w,v \in S(L(G))$, $\varphi(w) = \varphi(v)$ implies that $w = v$. It is clear that if $\varphi$ is injective, then $G$ is injective.
The converse is not true: consider $\varphi: a \to abc, b \to bc, c \to a$, then $\varphi$ is not injective as $\varphi(cb) = \varphi(a)$ but $G = ( \{a,b,c\}, \varphi, a )$ is injective since $cb \not \in S(L(G))$. If $\varphi$ is \textit{non-erasing}, i.e., $\varphi(a) \neq \varepsilon$ for all $a \in \A$, then we speak about \textit{propagating D0L system}, shortly PD0L.

Given a D0L system $G = (\A, \varphi, w)$ we say that the letter $a$ is \textit{bounded} (or also of \textit{rank zero}) if the set $\{ \varphi^n(a) \colon n \in \N \}$ is finite.
If a letter is not bounded, it is \textit{unbounded}.
We denote the set of all bounded letters by $\A_0$.
The system $G$ is \textit{pushy} if $S(L(G))$ contains infinitely many factors over $\A_0$.

A D0L system is \textit{repetitive} if for any $k \in \N$ there is a non-empty word $w$ such that $w^k$ is a factor. By~\cite{EhRo83}, any repetitive D0L system is \textit{strongly repetitive}, i.e., there is a non-empty word $w$ such that $w^k$ is a factor for all $k \in \N$.

\section{Definition of circularity}

In the literature, one can find two slightly different views of circularity.
Both these views can be expressed in terms of interpretations.
\begin{definition}
Let $G = (\A,\varphi, w)$ be a PD0L-system. A triplet $(p,v,s)$ where $p,s \in \A^*$ and $v = v_1\cdots v_n \in \A^+$ is an \textit{interpretation of a word $u \in S(L(G))$} if $\varphi(v) = pus$.
\end{definition}

The following definition of circularity is used in~\cite{MiSe}.
\begin{definition} \label[definition]{def:circular-system}
	 Let $G = (\A,\varphi, w)$ be a PD0L-system and let $(p,v,s)$ and $(p',v',s')$ be two interpretations of a non-empty word $u \in S(L(G))$ with $v = v_1\cdots v_n$, $v' = v'_0 \cdots v'_m$ and $u = u_1 \cdots u_\ell$.
	 
	 We say that $G$ is \emph{circular} with \emph{synchronization delay $D > 0$} if whenever we have 
	 $$
	 	|\varphi(v_1\cdots v_i)| - |p| > D \quad \text{and} \quad |\varphi(v_{i+1} \cdots v_n)| - |s| > D
	 $$
	 for some $1 \leq i \leq n$, then there is $1 \leq j \leq m$ such that
	 $$
	 	|\varphi(v_1\cdots v_{i-1}| - |p| = |\varphi(v'_1\cdots v'_{j-1})| - |p'|
	 $$
	 and $v_i = v'_j$ (see Figure~\ref{fig:circularity_1}).
\end{definition}
This definition says that a long enough word has unique $\varphi$-preimage except for some prefix and suffix shorter than a constant $D$. Note that if a D0L system $G = (\A,\varphi, w)$ contains arbitrarily long words with two different $\varphi$-preimages (i.e., for any $n > 0$ there are words $v$ and $u$ in $S(L(G))$ longer than $n$ with $\varphi(v) = \varphi(u)$) cannot be circular.

In~\cite{Ca94}, a circular D0L system with injective morphism is defined using the notion of synchronizing point (see Section 3.2 in~\cite{Ca94} for details). We give here an equivalent definition employing the notion of interpretation. 
\begin{definition}
Let $G = (\A,\varphi, w)$ be a PD0L-system. We say that two interpretations $(p,v,s)$ and $(p',v',s')$ of a word $u \in SL(G)$ are \textit{synchronized at position $k$} if there exist nonnegative indices $i$ and $j$ such that
$$
	\varphi(v_1\cdots v_i) = p u_1 \cdots u_k \quad \text{ and } \quad \varphi(v'_1\cdots v'_j) = p' u_1 \cdots u_k
$$
with $v = v_1\cdots v_n$, $v' = v'_0 \cdots v'_m$ and $u = u_1 \cdots u_\ell$  (see Figure~\ref{fig:circularity_2}).

We say that a word $u \in S(L(G))$ has a \textit{synchronizing point} at position $k$ with $0 \leq k \leq |u|$ if all its interpretations are pairwise synchronized at position $k$.
\end{definition} \label[definition]{def:weak-circularity}
By~\cite{Ca94}, a D0L system $G$ with injective morphism is circular if there is positive $D$ such that any $v$ from $S(L(G))$ longer than $2D$ has a synchronizing point.
This definition is equivalent to \cref{def:circular-system}.
However, the synchronizing point is defined for D0L systems with just non-erasing morphism and so we can omit the assumption of injectiveness in \Cref{def:circular-system}.
\begin{definition}
	A PD0L-system $G$ is called \emph{weakly circular} if there is a constant $D > 0$ such that any $v$ from $S(L(G))$ longer that $2D$ has a synchronizing point.
\end{definition}
As said above, if $G$ is injective, weak circularity is equivalent to circularity.
As the following example shows, this is not true for the non-injective case.
\begin{example}\label[example]{ex:weak-needed}
Consider the D0L system $G_1 = (\{a,b,c\}, \varphi_1, a)$ with the non-injective $\varphi_1:  a \to abca, b \to bc, c \to bc$.
This system is not circular as for all $m \in \N$ the word $(bc)^{2m}$ has two different preimages $(bc)^m$ and $(cb)^m$.
The corresponding interpretations, however, have synchronizing points for $m > 1$ at positions $2k$ for all  $0 \leq k \leq m$.
Moreover, one can easily check that $G_1$ is weakly circular.
\end{example}
So, circularity implies weak circularity but the converse is not true.
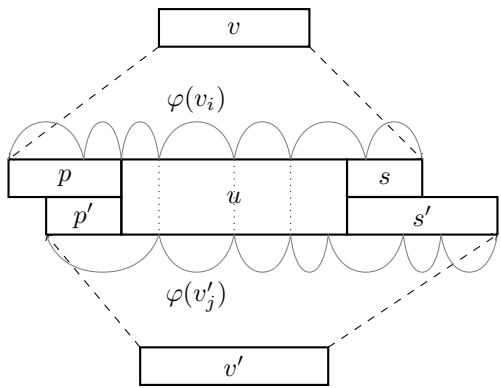
\begin{figure}[th]
\centering
\begin{tikzpicture}
	\draw[thick] (2.5,5) rectangle (4.5,5.5);	
	\node at (3.5,5.25) {$v$};
	\draw[thick] (.5,3.5) rectangle (2,3);
	\node at (1.25,3.25) {$p$};
	\draw[dashed,->] (2.5,5) -- (.5,3.5);
	\draw[thick] (2,3.5) rectangle (5,2.5);
	\node at (3.5,3) {$u$};
	\draw[thick] (5,3.5) rectangle (6,3);
	\node at (5.5,3.25) {$s$};
	\draw[dashed,->] (4.5,5) -- (6,3.5);
	\draw[thick] (1,3) rectangle (2,2.5);
	\node at (1.5,2.75) {$p'$};
	\draw[thick] (5,3) rectangle (7,2.5);
	\node at (6,2.75) {$s'$};
	\draw[thick] (2.25,0.5) rectangle (4.75,1);	
	\node at (3.5,.75) {$v'$};
	\draw[dashed,->] (2.25,1) -- (1,2.5);
	\draw[dashed,->] (4.75,1) -- (7,2.5);
	
	\draw [gray] plot [smooth, tension=1.5] coordinates { (0.5,3.5) (1,4) (1.5,3.5)};
	\draw [gray] plot [smooth, tension=1.5] coordinates { (1.5,3.5) (1.75,4) (2,3.5)};
	\draw [gray] plot [smooth, tension=1.5] coordinates { (2,3.5) (2.25,4) (2.5,3.5)};
	\draw [gray] plot [smooth, tension=1.5] coordinates { (2.5,3.5) (3,4) (3.5,3.5)};
	\node [above] at (3,4) {$\varphi(v_i)$};	
	\draw [gray] plot [smooth, tension=1.5] coordinates { (3.5,3.5) (3.875,4) (4.25,3.5)};
	\draw [gray] plot [smooth, tension=1.5] coordinates { (4.25,3.5) (4.75,4) (5.25,3.5)};
	\draw [gray] plot [smooth, tension=1.5] coordinates { (5.25,3.5) (5.625,4) (6,3.5)};	
	
	\draw [gray] plot [smooth, tension=1.5] coordinates { (1,2.5) (1.75,2) (2.5,2.5)};
	\draw [gray] plot [smooth, tension=1.5] coordinates { (2.5,2.5) (3,2) (3.5,2.5)};
	\node [below] at (3,2) {$\varphi(v'_j)$};
	\draw [gray] plot [smooth, tension=1.5] coordinates { (3.5,2.5) (3.875,2) (4.25,2.5)};	
	\draw [gray] plot [smooth, tension=1.5] coordinates { (4.25,2.5) (4.5,2) (4.75,2.5)};	
	\draw [gray] plot [smooth, tension=1.5] coordinates { (4.75,2.5) (5.25,2) (5.75,2.5)};
	\draw [gray] plot [smooth, tension=1.5] coordinates { (5.75,2.5) (6,2) (6.25,2.5)};
	\draw [gray] plot [smooth, tension=1.5] coordinates { (6.25,2.5) (6.625,2) (7,2.5)};
	
	\draw[dotted] (2.5,3.5) -- (2.5,2.5);
	\draw[dotted] (3.5,3.5) -- (3.5,2.5);
	\draw[dotted] (4.25,3.5) -- (4.25,2.5);

\end{tikzpicture}
\caption{Two interpretations from \cref{def:circular-system} with $v_i = v'_j$.}
 \label{fig:circularity_1}
\end{figure}

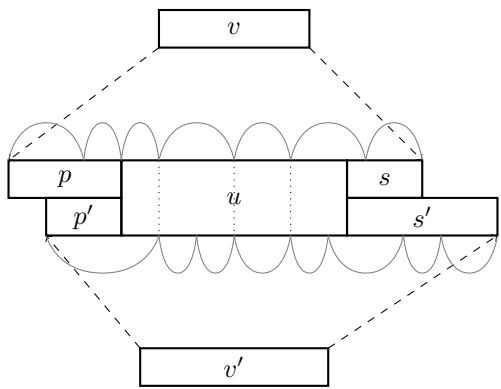
\begin{figure}[th]
\centering
\begin{tikzpicture}				
	\draw[thick] (11.5,5) rectangle (13.5,5.5);	
	\node at (12.5,5.25) {$v$};
	\draw[thick] (9.5,3.5) rectangle (11,3);
	\node at (10.25,3.25) {$p$};
	\draw[dashed,->] (11.5,5) -- (9.5,3.5);
	\draw[thick] (11,3.5) rectangle (14,2.5);
	\node at (12.5,3) {$u$};
	\draw[thick] (14,3.5) rectangle (15,3);
	\node at (14.5,3.25) {$s$};
	\draw[dashed,->] (13.5,5) -- (15,3.5);
	\draw[thick] (10,3) rectangle (11,2.5);
	\node at (10.5,2.75) {$p'$};
	\draw[thick] (14,3) rectangle (16,2.5);
	\node at (15,2.75) {$s'$};
	\draw[thick] (11.25,0.5) rectangle (13.75,1);	
	\node at (12.5,.75) {$v'$};
	\draw[dashed,->] (11.25,1) -- (10,2.5);
	\draw[dashed,->] (13.75,1) -- (16,2.5);
	
	\draw [gray] plot [smooth, tension=1.5] coordinates { (9.5,3.5) (10,4) (10.5,3.5)};
	\draw [gray] plot [smooth, tension=1.5] coordinates { (10.5,3.5) (10.75,4) (11,3.5)};
	\draw [gray] plot [smooth, tension=1.5] coordinates { (11,3.5) (11.25,4) (11.5,3.5)};
	\draw [gray] plot [smooth, tension=1.5] coordinates { (11.5,3.5) (12,4) (12.5,3.5)};
	\draw [gray] plot [smooth, tension=1.5] coordinates { (12.5,3.5) (12.875,4) (13.25,3.5)};
	\draw [gray] plot [smooth, tension=1.5] coordinates { (13.25,3.5) (13.75,4) (14.25,3.5)};
	\draw [gray] plot [smooth, tension=1.5] coordinates { (14.25,3.5) (14.625,4) (15,3.5)};	
	
	\draw [gray] plot [smooth, tension=1.5] coordinates { (10,2.5) (10.75,2) (11.5,2.5)};
	\draw [gray] plot [smooth, tension=1.5] coordinates { (11.5,2.5) (11.75,2) (12,2.5)};
	\draw [gray] plot [smooth, tension=1.5] coordinates { (12,2.5) (12.25,2) (12.5,2.5)};
	\draw [gray] plot [smooth, tension=1.5] coordinates { (12.5,2.5) (12.875,2) (13.25,2.5)};	
	\draw [gray] plot [smooth, tension=1.5] coordinates { (13.25,2.5) (13.5,2) (13.75,2.5)};	
	\draw [gray] plot [smooth, tension=1.5] coordinates { (13.75,2.5) (14.25,2) (14.75,2.5)};
	\draw [gray] plot [smooth, tension=1.5] coordinates { (14.75,2.5) (15,2) (15.25,2.5)};
	\draw [gray] plot [smooth, tension=1.5] coordinates { (15.25,2.5) (15.625,2) (16,2.5)};
	
	\draw[dotted] (11.5,3.5) -- (11.5,2.5);
	\draw[dotted] (12.5,3.5) -- (12.5,2.5);
	\draw[dotted] (13.25,3.5) -- (13.25,2.5);
	
\end{tikzpicture}
\caption{Two interpretations from \cref{def:weak-circularity} synchronized at positions depicted by dotted lines.}
 \label{fig:circularity_2}
\end{figure}

\section{Main result}

\begin{theorem}\label[theorem]{thm:main-result}
	Any PD0L system that is not weakly circular is repetitive.
\end{theorem}

The two following lemmas will be used to prove this theorem,.
The next lemma and its proof is based on the ideas in the proof of Theorem 4.35 in \cite{Ku_ToSyDy}.
\begin{lemma}\label[lemma]{lem:kurka}
	Let $G = (\A,\varphi,w)$ be a PD0L system.
	If there exists a sequence $\epsilon(k)$ with $\displaystyle \lim_{k \to +\infty} \epsilon(k) = +\infty$
	and if for any $k \in \N$ there are two non-empty words $u$ and $v$ in $S(L(G))$ containing an unbounded letter such that the following conditions are satisfied 
	\begin{enumerate}[(i)]
		\item $|u| = k$;
		\item \label{lk_bod2} there are two integers $m$ and $n$ such that $m > n$ and letters $a$ and $b$ such that for each $i \in \{m,n\}$ the word $\varphi^i(u)$ is a factor of $\varphi^i(v)$ and $\varphi^i(v)$ is a factor of $\varphi^i(aub)$,
		moreover, $\frac{|\varphi^i(u)|}{|\varphi^i(a)|} > \epsilon(k)$ or $\frac{|\varphi^i(u)|}{|\varphi^i(b)|} > \epsilon(k)$; and
		\item \label{lk_bod3} for each $i \in \{m,n\}$ the factor $\varphi^i(u)$ has no synchronizing point: two non-synchronized in\-ter\-pre\-ta\-ti\-ons are $(\varepsilon,\varphi^{i-1}(u),\varepsilon)$ and $(p_i,\varphi^{i-1}(v),s_i)$,
	\end{enumerate}
	then the D0L system is repetitive.
\end{lemma}

\begin{proof}
	Suppose that $\frac{|\varphi^i(u)|}{|\varphi^i(a)|} > \epsilon(k)$ is true in requirement \eqref{lk_bod2}, the other case $\frac{|\varphi^i(u)|}{|\varphi^i(b)|} > \epsilon(k)$ is analogous.
	It holds that
	$$
		\varphi^m(v) = p_m \varphi^m(u) s_m = \varphi^{m-n}(\varphi^n(v)) = \varphi^{m-n}(p_n) \varphi^m(u) \varphi^{m-n}(s_n).
	$$
	The fact that the interpretations $(\varepsilon,\varphi^{m-1}(u),\varepsilon)$ and $(p_m,\varphi^{m-1}(v),s_m)$ are not synchronized implies that $p_m \neq \varphi^{m-n}(p_n)$ (if $p_m = \varphi^{m-n}(p_n)$, the two interpretations of $\varphi^m(u)$ are synchronized at position $0$, see Figure~\ref{fig:proof_of_lemma_1}).
	Since $p_m \varphi^m(u) s_m = \varphi^{m-n}(p_n) \varphi^m(u) \varphi^{m-n}(s_n)$ the word $p_m$ is a proper prefix of $\varphi^{m-n}(p_n)$ or vice versa.
	Moreover, $p_m$ is not empty since it contradicts again the point \eqref{lk_bod3}.
	Suppose $p_m$ is a non-empty proper prefix of $\varphi^{m-n}(p_n)$.
	It implies there exists a word $z$ such that $p_m z = \varphi^{m-n}(p_n)$.
	If $\varphi^{m-n}(p_n)$ is a non-empty proper prefix of $p_m$, then we may find a word $z$ such that $\varphi^{m-n}(p_n)z = p_m$ (see Figure~\ref{fig:proof_of_lemma_2}).
\begin{figure}[th]
\centering
\begin{tikzpicture}
	\draw[thick] (2,5) rectangle (2.5,5.5);	
	\node at (2.25,5.25) {$a$};
	\draw[thick] (2.5,5) rectangle (4.5,5.5);	
	\node at (3.5,5.25) {$u$};
	\draw[thick] (4.5,5) rectangle (5,5.5);	
	\node at (4.75,5.25) {$b$};
	\draw[thick] (0,3.5) rectangle (2,3);
	\node at (1,3.25) {$p_m$};
	\draw[thick] (2,3.5) rectangle (5,3);
	\node at (3.5,3.25) {$\varphi^m(u)$};
	\draw[thick] (5,3.5) rectangle (7.5,3);
	\node at (6.25,3.25) {$s_m$};

	\draw[dashed,->] (2.5,5) -- (2,3.5);
	\draw[dashed,-] (2,5) -- (-0,4);	
	\draw[dashed,->] (4.5,5) -- (5,3.5);
	\draw[dashed,-] (5,5) -- (7.5,4);
	
	\draw[thick] (0,3) rectangle (2,2.5);
	\node at (1,2.75) {$\varphi^{m-n}(p_n)$};
	\draw[thick] (2,3) rectangle (5,2.5);
	\node at (3.5,2.75) {$\varphi^m(u)$};
	\draw[thick] (5,3) rectangle (7.5,2.5);
	\node at (6.25,2.75) {$\varphi^{m-n}(s_n)$};
	\node at (8.5,3) {$= \varphi^m(v)$};
	
	\draw[thick] (1.5,1) rectangle (2.5,0.5);
	\node at (2,0.75) {$p_n$};
	\draw[thick] (2.5,1) rectangle (4.5,0.5);
	\node at (3.5,.75) {$\varphi^n(u)$};
	\draw[thick] (4.5,1) rectangle (5.5,0.5);
	\node at (5,.75) {$s_n$};
	\node[right] at (5.5,0.75) {$= \varphi^n(v)$};
	
	\draw[dashed,->] (1.5,1) -- (0,2.5);
	\draw[dashed,->] (2.5,1) -- (2,2.5);
	\draw[dashed,->] (4.5,1) -- (5,2.5);
	\draw[dashed,->] (5.5,1) -- (7.5,2.5);
	
	\draw[thick] (2.75,-1.5) rectangle (4.25,-1);
	\node at (3.5,-1.25) {$v$};
	\draw[dashed,->] (2.75,-1) -- (1.5,.5);
	\draw[dashed,->] (4.25,-1) -- (5.5,.5);
\end{tikzpicture}
\caption{The first arrangement from the proof of \Cref{lem:kurka}.}
\label{fig:proof_of_lemma_1}
\end{figure}
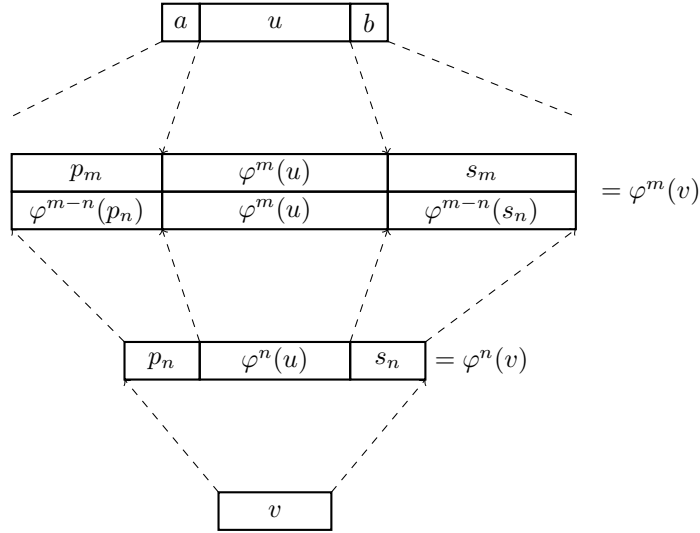

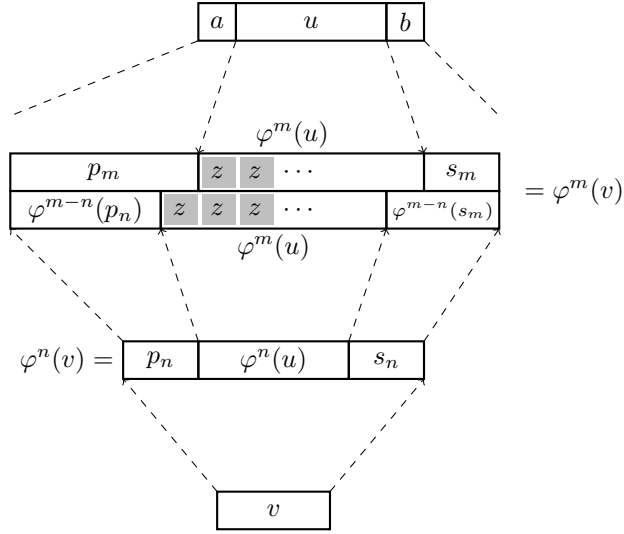
\begin{figure}[th]
\centering
\begin{tikzpicture}
	\draw[thick] (12,5) rectangle (12.5,5.5);	
	\node at (12.25,5.25) {$a$};
	\draw[thick] (12.5,5) rectangle (14.5,5.5);	
	\node at (13.5,5.25) {$u$};
	\draw[thick] (14.5,5) rectangle (15,5.5);	
	\node at (14.75,5.25) {$b$};
	\draw[thick] (9.5,3.5) rectangle (12,3);
	\node at (10.75,3.25) {$p_m$};
	\draw[thick] (12,3.5) rectangle (15,3);
	\node at (13.25,3.75) {$\varphi^m(u)$};
	\path [fill=lightgray] (12.05,3.45) rectangle (12.5,3.05);
	\node at (12.25, 3.25) {$z$};
	\path [fill=lightgray] (12.55,3.45) rectangle (13,3.05);
	\node at (12.75, 3.25) {$z$};
	\node [right] at (13, 3.25) {$\cdots$};
	\draw[thick] (15,3.5) rectangle (16,3);
	\node at (15.5,3.25) {$s_m$};	

	\draw[dashed,->] (12.5,5) -- (12,3.5);
	\draw[dashed,-] (12,5) -- (9.5,4);	
	\draw[dashed,->] (14.5,5) -- (15,3.5);
	\draw[dashed,-] (15,5) -- (16,4);
	
	\draw[thick] (9.5,3) rectangle (11.5,2.5);
	\node at (10.5,2.75) {$\varphi^{m-n}(p_n)$};
	\draw[thick] (11.5,3) rectangle (14.5,2.5);
	\node at (13,2.25) {$\varphi^m(u)$};
	\path [fill=lightgray] (11.55,2.95) rectangle (12,2.55);
	\node at (11.75, 2.75) {$z$};
	\path [fill=lightgray] (12.05,2.95) rectangle (12.5,2.55);
	\node at (12.25, 2.75) {$z$};
	\path [fill=lightgray] (12.55,2.95) rectangle (13,2.55);
	\node at (12.75, 2.75) {$z$};
	\node [right] at (13, 2.75) {$\cdots$};		
	\draw[thick] (14.5,3) rectangle (16,2.5);
	\node at (15.25,2.75) {$\scriptstyle{\varphi^{m-n}(s_m)}$};
	
	\draw[thick] (11,1) rectangle (12,0.5);
	\node at (11.5,0.75) {$p_n$};
	\draw[thick] (12,1) rectangle (14,0.5);
	\node at (13,.75) {$\varphi^n(u)$};
	\draw[thick] (14,1) rectangle (15,0.5);
	\node at (14.5,.75) {$s_n$};
	\node[left] at (11,0.75) {$\varphi^n(v) = $};
	
	\draw[dashed,->] (11,1) -- (9.5,2.5);
	\draw[dashed,->] (12,1) -- (11.5,2.5);
	\draw[dashed,->] (14,1) -- (14.5,2.5);
	\draw[dashed,->] (15,1) -- (16,2.5);
	
	\draw[thick] (12.25,-1.5) rectangle (13.75,-1);
	\node at (13,-1.25) {$v$};
	\draw[dashed,->] (12.25,-1) -- (11,.5);
	\draw[dashed,->] (13.75,-1) -- (15,.5);
	
	\node at (17,3) {$= \varphi^m(v)$};
	\end{tikzpicture}
\caption{The second arrangement from the proof of \cref{lem:kurka}.}
\label{fig:proof_of_lemma_2}
\end{figure}
	Therefore, in both cases, the word $\varphi^m(u)$ is a prefix of $z^\ell$ for some integer $\ell$. Since 
	$$
		\frac{|\varphi^m(u)|}{|z|} > \frac{|\varphi^m(u)|}{\max\{|p_m|,|\varphi^{m-n}(p_n)|\}} > \frac{|\varphi^m(u)|}{|\varphi^m(a)|} > \epsilon(k),
	$$
	we deduce that $z^{\lfloor \epsilon(k) \rfloor}$ is factor of $\varphi^m(u)$.

	As $\displaystyle \lim_{k \to +\infty} \epsilon(k) = +\infty$, the D0L-system is repetitive.	
\end{proof}

\begin{lemma} \label{le:bounded_are_sync}
	In any PD0L system there is a constant $C$ such that all factors over bounded letters longer than $C$ have a synchronizing point. 
\end{lemma}

\begin{proof}
	The statement is trivial for non-pushy D0L systems, hence we consider a pushy one. Clearly, there exist an integer $n$ such that for all $c \in \A_0$ we have $|\varphi^m(c)| = |\varphi^{m+1}(c)|$ for every $m \geq n$.
	Let $u$ be a factor over bounded letters only of length at least $L = 3 \| \varphi^{n+1} \|\cdot |w_0|$ where $w_0$ is the axiom of the D0L system. This implies that $u$ appears in the sequence $E(G)= (w_i)_{i \geq 0}$ in $w_k$ for $k > n+1$.
	
	Let $(p,w,s)$ be an interpretation of $u$. Since $u$ is a factor of $w_k$ such that $k > n+1$ and $|w_k| > L$, there must be words $x, y \in \A^*$ and $v \in \A^+$ such that $w = x\varphi^n(v)y$ and 
	$$
		|\varphi(x)| - |p| < \|\varphi^{n+1}\| \quad  \text{and} \quad |\varphi(y)| - |s| < \|\varphi^{n+1}\|.
	$$
	As $\varphi^{n+1}(v)$ is a factor of $u$, it contains only bounded letters, and thus so does the word $v$. Moreover, by the definition of $n$, every letter $c$ occurring in $\varphi^n(v)$ satisfies $|\varphi^n(c)| = |\varphi^{n+1}(c)|$.
	
	It follows that any two interpretations $(p,w,s)$ and $(p',w',s')$ of the word $u$ are synchronized at position $k = \|\varphi^{n+1}\|$.
	
\end{proof}

\begin{proof}[Proof of \cref{thm:main-result}]

Consider a PD0L system $G = (\A,\varphi, w_0)$ with infinite language (the statement for D0L system with finite language is trivial). We define a partition of the alphabet $\A = \Sigma_m \cup \Sigma_{m-1} \cup \cdots \cup \Sigma_1 \cup  \Sigma_0$ as follows:
\begin{enumerate}[(i)]
	\item $\Sigma_0 = \A_0$ is the set of bounded letters,
	\item if $x$ and $y$ are from $\Sigma_i$, then the sequence $\left(\frac{|\varphi^n(x)|}{|\varphi^n(y)|}\right)_{n \geq 1}$ is $\Theta(1)$,
	\item for all $i = 1, \ldots, m$, if $x$ is an element of $\Sigma_i$ and $y$ of $\Sigma_{i-1}$, then $\displaystyle \lim_{n \to + \infty} \frac{|\varphi^n(x)|}{|\varphi^n(y)|} = + \infty$.
\end{enumerate}
This partition is well defined due to~\cite{SaSo} where it is proved that for any $a \in \A$ there are numbers $\alpha \in \N$ and $\beta \in \R_{\geq 1} \cup \{0\}$ such that $|\varphi^n(a)| = \Theta(n^\alpha \beta^n)$. Further we define for all $j = 0, 1, \ldots, m$ the sets
$$
	\A_j = \bigcup_{0 \leq i \leq j} \Sigma_{i}.
$$
Note that $\varphi(\A_j) \subset \A_j$ and $\varphi(\A_j) \cap \Sigma_j \neq  \emptyset$.

Lemma \ref{le:bounded_are_sync} implies that factors without synchronizing point over $\A_0$ are bounded in length.
Fix a positive integer $j$ and assume that there is a factor without synchronizing point of arbitrary length over $\A_j$.
Let $k$ be a positive integer. For any positive $\ell \in \N$ we can find words $u^{(k)}_\ell \in \A_j^*$ and $v^{(k)}_\ell \in \A_j^*$ and letters $a^{(k)}_\ell \in \A_j$ and $b^{(k)}_\ell \in \A_j$ such that 
\begin{enumerate}[(a)]
	\item \label{du-1} $|u^{(k)}_\ell| = k$,
	\item \label{du-2} $\varphi^\ell(v^{(k)}_\ell)$ is a factor of $\varphi^\ell(a^{(k)}_\ell u^{(k)}_\ell b^{(k)}_\ell)$ and $\varphi^\ell(u^{(k)}_\ell)$ is a factor of $\varphi^\ell(v^{(k)}_\ell)$,
	\item $\varphi^\ell(u^{(k)}_\ell)$ has two non-synchronized interpretations $$(\varepsilon, \varphi^{\ell-1}(u^{(k)}_\ell), \varepsilon) \text{ and  } (p^{(k)}_\ell,\varphi^{\ell-1}(v^{(k)}_\ell),s^{(k)}_\ell)$$ where $p^{(k)}_\ell\varphi^{\ell}(u^{(k)}_\ell)s^{(k)}_\ell = \varphi^{\ell}(v^{(k)}_\ell)$.
\end{enumerate}
Since the length of $u^{(k)}_\ell$ is fixed, there must be an infinite set $E_1^{(k)} \subset \N$ such that $u^{(k)}_i = u^{(k)}_j = u^{(k)}$, $a^{(k)}_i = a^{(k)}_j = a^{(k)}$ and $b^{(k)}_i = b^{(k)}_j = b^{(k)}$ for all $i,j$ from $E_1^{(k)}$.

If for each $k$ there are indices $\ell_1 > \ell_2$ in $E_1^{(k)}$ such that $v^{(k)}_{\ell_1} = v^{(k)}_{\ell_2} = v^{(k)}$ and if the number of letters from $\Sigma_j$ in $u^{(k)}$ tends to $+\infty$ as $k \to +\infty$, then $G$ is repetitive by \cref{lem:kurka} and the proof is finished.

Assume no such indices $\ell_1, \ell_2$ exist for some $k$, then $|v^{(k)}_\ell|$ must go to infinity as $\ell \to +\infty$. It follows from \eqref{du-1} and \eqref{du-2} that the number of letters from $\Sigma_j$ in words $v^{(k)}_\ell$ is bounded (or even zero) and so there must be $j' \in \{ 1,\ldots, j-1\}$ such that the number of letters from $\Sigma_{j'}$ in $v^{(k)}_\ell$ goes to infinity as $\ell \to +\infty$ and there is a factor  without a synchronizing point over $\A_{j'}$ of arbitrary length. Note that such $j'$ must exist since number of letters from $\Sigma_j$ is bounded and the factors of $v^{(k)}_\ell$ containing only letters from $\A_0$ are bounded in length.

If such indices $\ell_1, \ell_2$ exist for each $k$ but the number of letters from $\Sigma_j$ in $u^{(k)}$ is bounded as $k \to +\infty$, there must be again some $j' \in \{ 1,\ldots, j-1\}$ such that the number of letters from $\Sigma_{j'}$ in $u^{(k)}$ goes to infinity as $k \to +\infty$ and there is again a factor without a synchronizing point over $\A_{j'}$ of arbitrary length.

Overall, given the integer $j$, we either prove $G$ is repetitive by \cref{lem:kurka} or we find a positive integer $j'$ less than $j$ such that there is a factor without a synchronizing point  over $\A_{j'}$ of arbitrary length. In the latter case we repeat the construction for $j = j'$.

The only remaining case is when $j = 1$, i.e. we have a factor  without a synchronizing point  over $\A_1$ of arbitrary length. 
Even in this case we can repeat the construction above. However, the case when $\ell_1 > \ell_2$ such that $v^{(k)}_{\ell_1} = v^{(k)}_{\ell_1} = v^{(k)}$ do not exists for some $k$ is not possible.

Indeed, it cannot happen that $|v^{(k)}_\ell|$ goes to infinity as $\ell \to +\infty$: $v^{(k)}_\ell$ must consist of letters from $\A_1 = \Sigma_1 \cup \Sigma_0$. 
Since $u^{(k)}$ is over $\A_1$ as well (with at least one letter from $\Sigma_1$ for $k$ large enough), 
the number of letters from $\Sigma_1$ in $v^{(k)}$ cannot be unbounded (for $\ell \to +\infty$) by the definition of $\Sigma_1$. Clearly, again by \cref{le:bounded_are_sync}, the number of letters from $\A_0$ in $v^{(k)}_\ell$ is bounded as well (for $\ell \to +\infty$) and
so the  indices $\ell_1 > \ell_2$ must exist so that $v^{(k)}_{\ell_1} = v^{(k)}_{\ell_1} = v^{(k)}$ . 

Moreover, since factors without a synchronizing point over $\A_0$ are bounded in length, the number of letters from $\Sigma_1$ in $u^{(k)}$ goes to infinity as $k \to +\infty$.

This all implies that $G$ is repetitive by \cref{lem:kurka}.
	
\end{proof}

\section{Simple criterion for circularity}

\begin{definition}
We say that a D0L system $G$ is \textit{unboundedly repetitive} if there exists $w \in S(L(G))$ such that
$w^k \in S(L(G))$ for all $k$ and $w$ contains at least one unbounded letter.
\end{definition}

In \cite{EhRo78}, the authors introduced the notion of simplification to study properties of a D0L system.
Given an endomorphism $\varphi$ over $\A$, the endomorphism $\Psi$ over $\B$ is its \textit{simplification}
if $\# \B < \# \A$ and there exist morphisms $h: \A^* \to \B^*$ and $k: \B^* \to \A^*$ such that $\varphi = kh$ and $\Psi = hk$.
A corollary of the defect theorem (see \cite{KoOt00}) is that every non-injective morphism has a simplification which is injective, called an \textit{injective simplification}. Specially, injective $G$ is its own injective simplification.

The following claim follows from Proposition 4.3 in \cite{KoOt00} and Theorem 2 in \cite{KlSt13}.
\begin{proposition} \label{unboundedly-repetitive-characterization}
A D0L system $G$ is unboundedly repetitive if and only if 
for some its injective simplification $G' = (\B, \psi, w'_0)$  of $G$ there is a positive integer $\ell $ and $a \in \B$ such that
			$$
				(\psi^\ell)^\infty(a) = w^\omega \quad \text{ for some } w \in \B^+.
			$$
\end{proposition}
In fact, if the condition in the previous claim is satisfied for some injective simplification, then it is satisfied for all injective simplifications.

Using this proposition and Theorem 1 of \cite{KlSt13} we deduce the following theorem.
\begin{theorem} \label{repetitive-pushy-or-unbounded}
	Let $G$ be a repetitive D0L system, then one of the following is true:
	\begin{enumerate}[(i)]
		\item $G$ is pushy,
		\item $G$ is unboundedly repetitive.
	\end{enumerate}
\end{theorem}

In the previous section we proved that any PD0L system that is not weakly circular is repetitive. 
The next theorem gives a characterization of injective circular D0L systems.

\begin{theorem} \label{thm:circ-iff-ur}
	An injective D0L system $G = (\A, \varphi, w)$ is not circular if and only if it is unboundedly repetitive.
\end{theorem}

\begin{proof}
$(\Rightarrow)$:
As an injective morphism is non-erasing, \Cref{thm:main-result} implies that $G$ is repetitive.
Thus, by \Cref{repetitive-pushy-or-unbounded}, $G$ is pushy or unboundedly repetitive.
Suppose it is pushy and not unboundedly repetitive.
Therefore, there exist an integer $N$ such that all repetitions $u^\ell$ where $\ell > N$ and $u \in S(L(G))$ are over bounded letters only, i.e., $u \in \A_0^+$.
From the proof of \Cref{thm:main-result} one can see that long enough non-synchronized factors contain longer and longer repetitions but these repetitions cannot be over bounded letters due to \Cref{le:bounded_are_sync} -- a contradiction.
	
$(\Leftarrow)$:
\Cref{unboundedly-repetitive-characterization} implies that there is a positive integer $\ell $ and a letter $a$ such that
$(\varphi^\ell)^\infty(a) = w^\omega$ for some $w \in \A^+$.
In \cite{KlSt13} it is proved that the word $w$ can be taken so that it contains the letter $a$ only once at its beginning.
It follows that $\varphi^\ell(w) = w^k$ for some $k > 1$.
Since $\varphi$ is injective, we must have $\varphi(p) \neq w$ for all prefixes $p$ of $w$.
This implies that for all $n \in \N$ the word $w^{nk}$ has two non-synchronized interpretations $(\varepsilon, w^n, \varepsilon)$ and $(w,w^{n+1},w^{k-1})$.
\end{proof}

\begin{remark}
	In the previous theorem, we cannot omit assumption of injectiveness and replace circularity with weak circularity: consider again the D0L system $G_1$ from \cref{ex:weak-needed}. The conditions of \Cref{unboundedly-repetitive-characterization} is satisfied for $\ell = 1$ and the letter $b$ with $w = bc$ but still the corresponding D0L system is weakly circular.
\end{remark}

Since the existence of $\ell$ and $a$ satisfying conditions of \Cref{unboundedly-repetitive-characterization} can be tested by a simple and fast algorithm described in~\cite{La91}, we have a simple algorithm deciding circularity.

As a corollary of \Cref{thm:circ-iff-ur}, we retrieve the following result of \cite{Mo96}.
A morphism $\varphi: \A^* \to \A^*$ is \textit{primitive} if there exists an integer $k$ such that for all letters $a,b \in \A$, the letter $b$ appears in $\varphi^k(a)$.
An infinite word $\uu$ is a \textit{periodic point} of a morphism $\varphi$ if there exists an integer $k$ such that $\varphi^k(\uu) = \uu$.

\begin{corollary}[\cite{Mo96}]
	If $\uu$ is an aperiodic fixed point of a primitive morphism injective on $S(L(G))$, then it is circular.
\end{corollary}

\begin{proof}
Any periodic point of a primitive morphism has the same language as $\uu$.
Therefore, every periodic point is aperiodic and so the condition of \Cref{unboundedly-repetitive-characterization} cannot be satisfied.
\Cref{thm:circ-iff-ur} yields the result.
\end{proof}

\section*{Acknowledgements}

This work was supported by the Czech Science Foundation grant GA\v CR 13-35273P.

\bibliographystyle{alpha}
\bibliography{biblio}

\end{document}